\documentclass[12pt]{article}

\usepackage{amsmath, amsthm, amssymb, amsfonts}
\usepackage{mathrsfs}
\usepackage[margin=1in]{geometry}
\usepackage[colorlinks=true, linkcolor=blue, citecolor=blue,
            urlcolor=blue]{hyperref}
\usepackage{enumitem}
\usepackage{booktabs}
\usepackage{cite}

\theoremstyle{plain}
\newtheorem{theorem}{Theorem}[section]
\newtheorem{lemma}[theorem]{Lemma}
\newtheorem{corollary}[theorem]{Corollary}

\theoremstyle{definition}
\newtheorem{definition}[theorem]{Definition}

\theoremstyle{remark}
\newtheorem{remark}[theorem]{Remark}

\newcommand{\R}{\mathbb{R}}
\newcommand{\N}{\mathbb{N}}
\newcommand{\No}{\mathbb{N}_0}

\newcommand{\dee}{\mathrm{d}}
\newcommand{\e}{\mathrm{e}}

\newcommand{\dfact}[2]{(#1)_{#2,\lambda}}
\newcommand{\expdeg}{\e_\lambda}
\newcommand{\Tdeg}[1]{T_{#1,\lambda}[\expdeg]}
\newcommand{\Rdeg}[1]{R_{#1,\lambda}[\expdeg]}

\title{\bfseries Threshold Phenomena and Bounds in Normalized
  Remainders of Degenerate Exponential Functions}

\author{Artatrana Suna and Prasanta Kumar Ray$^{}$%
  \footnote{Corresponding author}\\[6pt]
  \small Department of Mathematics, Sambalpur University,\\[-2pt]
  \small Jyoti Vihar, Burla, Sambalpur, Odisha 768019, India\\[4pt]
  \small\texttt{suna.19972702@gmail.com},\quad
  \texttt{prasantamath@suniv.ac.in}}

\date{}

\begin{document}

\maketitle
\thispagestyle{empty}

\noindent\rule{\textwidth}{0.4pt}
\begin{abstract}
\noindent In this work, we study a normalized remainder $T_{n,\lambda}[\e_\lambda]$ for the
degenerate exponential $\e_\lambda(u)=(1+\lambda u)^{1/\lambda}$ ($\lambda>0$).
We establish an integral representation, an exact monotonicity threshold at
$\lambda=1/(n+1)$, and rigorous conditions for the local failure of logarithmic
convexity at the origin. We then prove a sharp asymptotic result: for every
$\lambda$ in the increasing regime $(0,1/(n+1))$, the second logarithmic
derivative satisfies $u^2L(u)\to -\alpha<0$ as $u\to\infty$, showing that
global logarithmic convexity on $(0,\infty)$ fails throughout this regime. We further give a necessary and sufficient condition for absolute monotonicity, showing it holds only on a countable, measure-zero set of parameters, and we derive explicit two-sided truncation-error bounds that are pointwise sharp at the origin.


\end{abstract}
\noindent\rule{\textwidth}{0.4pt}
\bigskip

%% ==================================================================
\section{Introduction}
%% ==================================================================

The study of remainders in Taylor expansions of $\e^{u}$ has a rich
history with applications in approximation theory, inequalities, and
special functions. The systematic investigation of normalized remainders
was initiated by Qi~\cite{Qi2024}, who established
absolute monotonicity of the normalized tail of the Maclaurin series
of $\e^u$. Subsequently, the positivity, monotonicity, logarithmic convexity, and absolute monotonicity were proven for the normalized remainders $T_n[\e^{u}]$ in~\cite{LiQi2025}. Further properties of the normalized tails were established in~\cite{BaoAgarwalQiDu2024}.
This program has since become an active research area; relevant
contributions include the study of normalized remainders of
the cosine and sine functions~\cite{WanQi2024,WangQi2024Sine},
of the tangent function and its
square~\cite{ZhangQi2025Tan,LiuLongQi2025},
of the generating function of the Bernoulli
numbers~\cite{ZhangYangQi2025},
and of a function arising from the reciprocal of the gamma
function~\cite{ZhangQi2025Gamma}.
A monotonicity result for the ratio of two normalized remainders of
the exponential function was established by Zhang and
Qi~\cite{ZhangQi2025}, and a systematic survey of the entire
normalized-remainder program, including open problems, is given
in~\cite{QiSurvey2025}.
The background on logarithmic convexity and complete monotonicity of functions involving
exponential functions and ratios of gamma functions is treated in depth in~\cite{QiGuo2021JMAA}.

Degenerate versions of special functions, obtained by replacing
$\e^{u}$ with $(1+\lambda u)^{1/\lambda}$, have attracted considerable
interest in combinatorics and number theory since Carlitz's
foundational work~\cite{Carlitz1956,Carlitz1979}. The degenerate exponential $\e_\lambda(u)=(1+\lambda u)^{1/\lambda}$ is
the foundational object of the theory. It generates the degenerate falling
factorials and underlies the construction of degenerate analogues of most
classical polynomial families~\cite{KimKim2020JMAA,KimKim2022}.
A substantial literature on degenerate Bernoulli, Euler, Stirling,
Bell, and related objects has developed in the past decade;
representative contributions appear
in~\cite{KimKimKwonLee2020,KimKim2021Sheffer,
HernandezPeraltaQuintana2023,Cesarano2024,ChenKimKimLeeLee2024}.
Properties of the degenerate Laplace transform and degenerate gamma
function were studied in~\cite{KimKim2017Laplace}.

In this paper we bridge these two active areas. We define a normalized
remainder for the degenerate exponential function and investigate its
analytic properties.
The following features are structurally new to the degenerate setting
and have no classical parallel. The monotonicity threshold at the
$n$-dependent value $\lambda=1/(n+1)$ (beyond which $\Tdeg{n}$ becomes
strictly decreasing), the complete failure of global logarithmic
convexity for every $\lambda$ in the increasing regime $(0,1/(n+1))$,
the restriction of absolute monotonicity to a countable parameter set,
and the power-type rather than exponential asymptotic growth.
The basic definitions and the existence of monotonicity and convexity
thresholds are natural analogues of the classical theory. However,
the precise location of these $n$-dependent thresholds, and the question
of whether logarithmic convexity can be recovered globally for small
$\lambda$ in the increasing regime $(0,1/(n+1))$, require new analysis.
We answer the latter question negatively for every $\lambda$ in the
increasing regime (Theorem~\ref{thm:asymptotic-log-concave},
Corollary~\ref{cor:unified-log-concave}). Unlike the classical case,
where $T_n[\e^u]$ is globally logarithmically convex on $(0,\infty)$
for every $n$, the degenerate remainder $\Tdeg{n}$ fails global
logarithmic convexity on $(0,\infty)$ for every
$\lambda\in(0,1/(n+1))$.
The asymptotic growth of the remainder is of power type,
$\Tdeg{n}(u)\sim Cu^{\alpha}$ as $u\to\infty$, in contrast to the
exponential-type growth of the classical case. Moreover, the decreasing
regime $\lambda>1/(n+1)$ and the boundary $\lambda=1/(n+1)$ are
discussed separately in Section~\ref{sec:log-convexity}.

The main contributions of our study are as follows.
\begin{enumerate}
\item An integral representation for the normalized degenerate
  remainder (Theorem~\ref{thm:integral-rep}).
\item An exact monotonicity threshold at $\lambda=1/(n+1)$
  (Theorem~\ref{thm:monotonicity-threshold}).
\item Conditions for local failure of logarithmic convexity at the
  origin (Theorem~\ref{thm:log-convex-failure}) and an asymptotic
  theorem showing global failure throughout the increasing regime
  at rate $u^2L(u)\to-\alpha$ (Theorem~\ref{thm:asymptotic-log-concave},
  Corollary~\ref{cor:unified-log-concave}).
\item A necessary and sufficient condition for absolute monotonicity
  (Theorem~\ref{thm:abs-monotonicity}), and the consequence that
  absolute monotonicity holds only on a countable, measure-zero
  parameter set (Corollary~\ref{cor:ae-failure}).
\item Two-sided bounds pointwise sharp at the origin
  (Theorem~\ref{thm:bounds}), with explicit truncation-error
  estimates for $\expdeg$ (Corollary~\ref{cor:truncation-error}).
\item Recovery of classical properties in the $\lambda\to0^+$ limit
  (Theorem~\ref{thm:classical-limit}).
\end{enumerate}

Table~\ref{tab:comparison} summarises the key properties
of the normalized remainder in the classical and degenerate settings.

\begin{table}[ht]
\centering
\renewcommand{\arraystretch}{1.25}
\begin{tabular}{lcc}
\toprule
Property & Classical ($T_n[\e^u]$) & Degenerate ($\Tdeg{n}$, $\lambda>0$) \\
\midrule
Positivity on $(0,\infty)$
  & Always & Always (Corollary~\ref{cor:positivity}) \\
Increasing on $(0,\infty)$
  & Always & Iff $\lambda<\tfrac{1}{n+1}$ (Theorem~\ref{thm:monotonicity-threshold}) \\
Global log-convexity on $(0,\infty)$
  & Always & Never in $(0,\tfrac{1}{n+1})$; trivial at $\tfrac{1}{n+1}$ \\
  & & (Corollary~\ref{cor:unified-log-concave}) \\
Asymptotic growth
  & Exponential ($\sim e^u$) & Power-type ($\sim Cu^\alpha$) \\
Absolute monotonicity
  & Always (Qi~\cite{Qi2024}) & Countable, measure-zero set of $\lambda$ \\
  & & (Theorem~\ref{thm:abs-monotonicity}) \\
Upper bound (for $u>0$)
  & $\e^u$ & $(1+\lambda u)^{\alpha_n}$ (Theorem~\ref{thm:bounds}) \\
\bottomrule
\end{tabular}
\caption{Comparison of the classical normalized remainder $T_n[\e^u]$
and the degenerate remainder $\Tdeg{n}=T_{n,\lambda}[\e_\lambda]$.
The ``increasing'' and ``log-convexity'' rows refer to the increasing
regime $\lambda\in(0,1/(n+1))$ for the degenerate case.}
\label{tab:comparison}
\end{table}

%% ==================================================================
\section{Preliminaries}\label{sec:preliminaries}
%% ==================================================================

We write $\N=\{1,2,3,\dotsc\}$ and $\No=\{0,1,2,\dotsc\}$.
For $\lambda>0$, the \emph{degenerate exponential function} is
\[
  \expdeg(u)=(1+\lambda u)^{1/\lambda},\qquad u\in(-1/\lambda,\infty).
\]
As $\lambda\to0^+$, $\expdeg(u)\to\e^u$, recovering the classical
exponential. The degenerate exponential plays a central role in the
theory of degenerate polynomial families; we refer to~\cite{KimKim2020JMAA,KimKim2021Sheffer} for its use in defining
the degenerate polyexponential functions and the degenerate Sheffer
sequence framework.

\begin{remark}
  Throughout this paper we take $\lambda>0$. The case $\lambda<0$
  is structurally different, not merely a sign change. For
  $\lambda<0$ the positivity constraint $1+\lambda u>0$ forces
  $u<-1/\lambda=1/|\lambda|$, so the natural domain of $\expdeg$
  becomes the left-unbounded interval
  $(-\infty,\,1/|\lambda|)$. Consequently, several arguments
  in this paper that rely on letting $u\to+\infty$ within the domain,
  most notably the asymptotic analysis underlying
  Theorem~\ref{thm:asymptotic-log-concave}, do not carry over to
  $\lambda<0$ without modification.
\end{remark}

\begin{definition}\label{def:factorials}
  For $x\in\R$ and $n\in\No$, the \emph{degenerate falling factorial}
  is defined as~\cite{Carlitz1979}
  \[
    \dfact{x}{n}=\prod_{k=0}^{n-1}(x-k\lambda).
  \]
See also~\cite{KimKim2020JMAA,KimKim2022} for its role in the
modern theory of degenerate special functions.
\end{definition}

\begin{lemma}[Series expansion {\cite{Carlitz1956,Carlitz1979,KimKim2020JMAA}}]
\label{lem:series-expansion}
  For $u\in(-1/\lambda,1/\lambda)$, i.e.\ $|\lambda u|<1$,
  \[
    (1+\lambda u)^{x/\lambda}
    =\sum_{j=0}^{\infty}\dfact{x}{j}\,\frac{u^{j}}{j!}.
  \]
  In particular, $\expdeg(u)=\sum_{j=0}^{\infty}\dfact{1}{j}\,
  u^{j}/j!$ for $|\lambda u|<1$.
\end{lemma}

\begin{proof}
The generalized binomial theorem gives $(1+y)^r=\sum_{j=0}^\infty\binom{r}{j}y^j$
for $|y|<1$, where $\binom{r}{j}=r(r-1)\cdots(r-j+1)/j!$.
Setting $y=\lambda u$ and $r=x/\lambda$, the $j$-th binomial coefficient
becomes $\binom{x/\lambda}{j}=\dfact{x}{j}/(\lambda^j\,j!)$, so the
$j$-th term equals $\dfact{x}{j}\,u^j/j!$, valid for $|\lambda u|<1$.
\end{proof}
\begin{remark}
The radius of convergence for $\expdeg(u)$ in $u$ is $1/\lambda$.
Convergence on the boundary $|\lambda u|=1$ depends on $x/\lambda$
and is not considered here. Note that $\expdeg(u)$ is defined on the
strictly larger domain $(-1/\lambda,\infty)$; the series represents
$\expdeg$ only on $|\lambda u|<1$.
\end{remark}

The following lemma establishes, once and for all, the uniform bound
needed to justify differentiation under the integral sign. It is
required in the proofs of Theorems~\ref{thm:monotonicity-threshold},
\ref{thm:log-convex-failure}, and~\ref{thm:abs-monotonicity}.

\begin{lemma}[Uniform domination lemma]\label{lem:domination}
For a fixed $\lambda>0$ and $[a,b]\subset(-1/\lambda,\infty)$, let
$m=\min(1,\,1+\lambda a)>0$ and $M=\max(1,\,1+\lambda b)<\infty$.
  Then $0<m\le1+\lambda uv\le M<\infty$ for all $u\in[a,b]$ and
  $v\in[0,1]$. Consequently, for every $\beta\in\R$,
  \[
    \bigl|(1+\lambda uv)^{\beta}\bigr|\le K_\beta<\infty
    \qquad\text{for all }u\in[a,b],\ v\in[0,1],
  \]
  where $K_\beta=M^{\beta}$ if $\beta\ge0$ and $K_\beta=m^{\beta}$
  if $\beta<0$.
\end{lemma}

\begin{proof}
  Since $a>-1/\lambda$, $1+\lambda a>0$, so $m>0$; clearly $M<\infty$.
  For each fixed $v\in[0,1]$, the map $u\mapsto1+\lambda uv$ is
  linear in $u$ with slope $\lambda v\ge0$, hence non-decreasing in $u$.
  Thus, its minimum and maximum
  over $u\in[a,b]$ are attained respectively at $u=a$ and $u=b$.
  Both boundary curves $v\mapsto1+\lambda av$ and
  $v\mapsto1+\lambda bv$ are linear in $v$, hence attain their
  extrema on $[0,1]$ at the endpoints $v=0,1$. Since the map
  $v\mapsto1+\lambda av$ is positive at both endpoints, we have
  \[
  \min_{v\in[0,1]}(1+\lambda av)=\min(1,1+\lambda a)=m
  \]
  and similarly
  \[
  \max_{v\in[0,1]}(1+\lambda bv)=\max(1,1+\lambda b)=M.
  \]
  Combining the above inequalities gives $0<m\le1+\lambda av\le1+\lambda uv\le1+\lambda bv\le
  M<\infty$ for all $u\in[a,b]$, $v\in[0,1]$.

  Finally, the map $s\mapsto s^\beta$ is increasing on $(0,\infty)$
  if $\beta\ge0$, giving $(1+\lambda uv)^\beta\le M^\beta$, and
  decreasing if $\beta<0$, giving $(1+\lambda uv)^\beta\le m^\beta$.
\end{proof}

%% ==================================================================
\section{Definition and Integral Representation}
\label{sec:definition}
%% ==================================================================

For $n\in\No$, let
  \[
    \Lambda_n=\bigl\{1/k\mid k\in\N,\;1\le k\le n\bigr\}
  \]
  denote the set of values of $\lambda$ at which
  $\dfact{1}{n+1}=\prod_{k=0}^{n}(1-k\lambda)$ vanishes.
  Note that $\Lambda_0=\emptyset$.

\begin{definition}[Degenerate normalized remainder]
  \label{def:degenerate-remainder}
  For $\lambda>0$ with $\lambda\notin\Lambda_n$ and
  $u\in(-1/\lambda,\infty)$, define
  \[
    \Tdeg{n}(u):=\begin{cases}
      \displaystyle
      \frac{(n+1)!}{\dfact{1}{n+1}}\cdot
      \frac{\Rdeg{n}(u)}{u^{n+1}},&u\ne0,\\[14pt]
      1,&u=0,
    \end{cases}
  \]
  where $\Rdeg{n}(u)=\expdeg(u)-\sum_{j=0}^{n}\dfact{1}{j}\,u^{j}/j!$
  is the $n$th remainder of the degenerate exponential series.
  Equivalently, for $u\ne0$ and $\lambda\notin\Lambda_n$,
  \[
    \Rdeg{n}(u)=\frac{\dfact{1}{n+1}}{(n+1)!}\,u^{n+1}\,\Tdeg{n}(u),
  \]
  so $\Tdeg{n}$ is the remainder $\Rdeg{n}$ divided by its leading
  Taylor coefficient $\dfact{1}{n+1}/(n+1)!$ and by $u^{n+1}$,
  normalised so that $\Tdeg{n}(0)=1$.
\end{definition}

\begin{remark}[Continuation in $\lambda$]
  \label{rem:analytic-continuation}
  The algebraic formula in Definition~\ref{def:degenerate-remainder}
  is valid only for $\lambda\notin\Lambda_n$; at $\lambda\in\Lambda_n$
  (when $n\ge1$) the factor $\dfact{1}{n+1}$ vanishes and the ratio
  presents a $0/0$ indeterminacy. The integral
  representation established in Theorem~\ref{thm:integral-rep} below
  is well-defined and continuous in $\lambda$ for all $\lambda>0$
  and provides the continuation to $\lambda\in\Lambda_n$.
  Since every subsequent result uses only this integral representation,
  no separate argument is needed at the exceptional values.
\end{remark}

\begin{theorem}[Integral representation]\label{thm:integral-rep}
  For $\lambda\notin\Lambda_n$ and $u\in(-1/\lambda,\infty)$, the
  quantity $\Tdeg{n}(u)$
  satisfies
  \begin{equation}\label{eq:int-rep-formula}
    \Tdeg{n}(u)
    =(n+1)\int_0^1(1-v)^n\,(1+\lambda uv)^{\frac{1}{\lambda}-n-1}
    \dee v.
  \end{equation}
\end{theorem}

\begin{proof}
  Let $f = \expdeg$. Using the chain rule and induction on $k$, one
  verifies that
   \begin{equation}\label{eq:kth-deriv}
    f^{(k)}(u)=\dfact{1}{k}\,(1+\lambda u)^{1/\lambda-k},
    \quad k\in\No.
  \end{equation}

  Since $1+\lambda u>0$ on $(-1/\lambda,\infty)$, the derivative
  $f^{(n+1)}(u)=\dfact{1}{n+1}(1+\lambda u)^{1/\lambda-n-1}$ is
  continuous there, so $f\in C^{n+1}(-1/\lambda,\infty)$. For
  any fixed $u\in(-1/\lambda,\infty)$, the integrand
  $(u-t)^n f^{(n+1)}(t)$ is continuous on the closed interval
  with endpoints $0$ and $u$. Thus, the integral is a proper Riemann
  integral. Taylor's theorem with integral remainder gives
  \begin{equation}\label{eq:taylor-remainder}
    \Rdeg{n}(u)
    =\frac{1}{n!}\int_0^u(u-t)^n f^{(n+1)}(t)\dee t
    =\frac{\dfact{1}{n+1}}{n!}
    \int_0^u(u-t)^n(1+\lambda t)^{1/\lambda-n-1}\dee t.
  \end{equation}

  For $\lambda\notin\Lambda_n$ (so $\dfact{1}{n+1}\ne0$) and
  $u\ne0$, cancelling $\dfact{1}{n+1}$ yields
  \[
    \Tdeg{n}(u)
    =\frac{n+1}{u^{n+1}}
    \int_0^u(u-t)^n(1+\lambda t)^{1/\lambda-n-1}\dee t.
  \]
  We apply the substitution $t=uv$, $v\in[0,1]$. Let
  $h(t)=(u-t)^n(1+\lambda t)^{1/\lambda-n-1}$, which is
  continuous on the closed interval with endpoints $0$ and $u$.
  The general substitution formula
  $\int_0^u h(t)\,\dee t=\int_0^1 h(uv)\,u\,\dee v$, valid for any
  $u\ne0$ (when $u<0$ both sides are negative for $h>0$, since
  the left side integrates in the negative direction and the right
  side carries the factor $u<0$), applies directly. Since
  $u-t=u(1-v)$,
  \begin{align*}
    \int_0^u(u-t)^n(1+\lambda t)^{1/\lambda-n-1}\dee t
    &= \int_0^1 u^n(1-v)^n(1+\lambda uv)^{1/\lambda-n-1}\cdot u\,\dee v\\
    &= u^{n+1}\int_0^1(1-v)^n(1+\lambda uv)^{1/\lambda-n-1}\dee v.
  \end{align*}
  Substituting back and cancelling $u^{n+1}\ne0$ yield
  \[
    \Tdeg{n}(u)
    =(n+1)\int_0^1(1-v)^n(1+\lambda uv)^{1/\lambda-n-1}\dee v.
  \]
  For all $u\in(-1/\lambda,\infty)$ and $v\in[0,1]$,
  $1+\lambda uv>0$, so the right-hand side
  of~\eqref{eq:int-rep-formula} is well-defined and continuous in
  $\lambda>0$, including at $\lambda\in\Lambda_n$, where it serves
  as the continuation described in
  Remark~\ref{rem:analytic-continuation}.

  For $u=0$, using the Beta identity $\int_0^1(1-v)^n\dee v
  =B(1,n+1)=\frac{1}{n+1}$, we have
  \[
    \lim_{u\to0}\Tdeg{n}(u)
    =(n+1)\int_0^1(1-v)^n\dee v
    =(n+1)\cdot\frac{1}{n+1}=1=\Tdeg{n}(0).\qedhere
  \]
\end{proof}
\begin{remark}\label{rem:primary-defn}
  Formula~\eqref{eq:int-rep-formula} may be taken as
  the primary definition of $\Tdeg{n}(u)$ for all $\lambda>0$, as it
  also works for $\lambda\in\Lambda_n$ unlike
  Definition~\ref{def:degenerate-remainder}.
\end{remark}
\begin{corollary}[Positivity]\label{cor:positivity}
  For all $\lambda>0$, $n\in\No$, and $u\in(-1/\lambda,\infty)$,
  $\Tdeg{n}(u)>0$.
\end{corollary}
\begin{proof}
  The integrand $(1-v)^n(1+\lambda uv)^{1/\lambda-n-1}$ is
  non-negative on $[0,1]$ and strictly positive for $v\in(0,1)$.
  Since it is strictly positive on a set of positive measure, the
  integral is strictly positive.
\end{proof}

%% ==================================================================
\section{Monotonicity and Its Threshold}\label{sec:monotonicity}
%% ==================================================================
Throughout this paper, for a given $\lambda>0$
and $n\in\No$, we set
\[
  \alpha\;=\;\alpha(\lambda,n)\;:=\;\frac{1}{\lambda}-n-1.
\]
In the bounds (Section~\ref{sec:bounds}) this quantity is written
$\alpha_n$ to emphasise its dependence on $n$; elsewhere we simply
write $\alpha$. The two notations refer to the same quantity.

\begin{theorem}[Monotonicity threshold]
  \label{thm:monotonicity-threshold}
  For $u\in(-1/\lambda,\infty)$,
  \[
    \frac{\dee}{\dee u}\Tdeg{n}(u)
    =(n+1)\bigl(1-(n+1)\lambda\bigr)
    \int_0^1 v(1-v)^n(1+\lambda uv)^{\alpha-1}\dee v.
  \]
  The sign of $\frac{\dee}{\dee u}\Tdeg{n}(u)$ is therefore
  determined by the sign of $1-(n+1)\lambda$, giving the following
  threshold classification.
  \begin{enumerate}
  \item If $\lambda<\frac{1}{n+1}$, then $\Tdeg{n}$ is strictly
    increasing.
  \item If $\lambda=\frac{1}{n+1}$, then $\Tdeg{n}(u)\equiv1$.
  \item If $\lambda>\frac{1}{n+1}$, then $\Tdeg{n}$ is strictly
    decreasing.
  \end{enumerate}
\end{theorem}

\begin{proof}
  Fix any compact $[a,b]\subset(-1/\lambda,\infty)$, and set
  $\phi(u,v)=(1-v)^n(1+\lambda uv)^{\alpha}$, whose partial
  derivative with respect to $u$ is
  $\partial_u\phi=\alpha\lambda v(1-v)^n(1+\lambda uv)^{\alpha-1}$.
  By Lemma~\ref{lem:domination} with $\beta=\alpha-1$, there exists
  a finite $K=K_{\alpha-1}$ such that
  $|(1+\lambda uv)^{\alpha-1}|\le K$ uniformly on $[a,b]\times[0,1]$.
  Hence $|\partial_u\phi|\le|\alpha\lambda|Kv(1-v)^n\in L^1([0,1])$,
  uniformly in $u$. The conditions of Leibniz's rule are satisfied, so
  \[
    \frac{\dee}{\dee u}\Tdeg{n}(u)
    =(n+1)\alpha\lambda
    \int_0^1 v(1-v)^n(1+\lambda uv)^{\alpha-1}\dee v,
  \]
  and $\alpha\lambda=1-(n+1)\lambda$ gives the stated formula.

  \textbf{Cases~1 and~3.}
  For any fixed $u\in(-1/\lambda,\infty)$ and $v\in(0,1)$, every
  factor in the integrand $v(1-v)^n(1+\lambda uv)^{\alpha-1}$ is
  strictly positive (since $1+\lambda uv>0$ on the domain), so the
  integral is strictly positive for every such $u$. The sign of
  the derivative is therefore determined solely by the $u$-independent
  prefactor $1-(n+1)\lambda$: positive when $\lambda<1/(n+1)$,
  negative when $\lambda>1/(n+1)$.

  \textbf{Case~2.}
  When $\lambda=1/(n+1)$, $\alpha=0$, and by the Beta
  identity $\int_0^1(1-v)^n\dee v=\frac{1}{n+1}$,
  \[
    \Tdeg{n}(u)=(n+1)\int_0^1(1-v)^n(1+\lambda uv)^{0}\dee v
    =(n+1)\cdot\frac{1}{n+1}=1.\qedhere
  \]
\end{proof}

%% ==================================================================
\section{Breakdown of Logarithmic Convexity}
\label{sec:log-convexity}
%% ==================================================================

Throughout this section, write $f=\Tdeg{n}$ for brevity, and define
the \emph{logarithmic second derivative}
\[
  L(u)=\frac{\dee^2}{\dee u^2}\ln f(u)
  =\frac{f''(u)f(u)-[f'(u)]^2}{[f(u)]^2}.
\]
This is well-defined since $f(u)>0$ everywhere
by Corollary~\ref{cor:positivity}. We call $f$ \emph{globally
logarithmically convex} (\emph{concave}) on an interval $I$
if $L(u)\ge0$ ($L(u)\le0$) for all $u\in I$, and
\emph{locally logarithmically convex at $u_0$} if $L(u_0)\ge0$.
This convention is fixed for the rest of the paper.

\begin{theorem}[Phase transitions in logarithmic convexity]
  \label{thm:log-convex-failure}
  Let $f(u)=\Tdeg{n}(u)$ and $L(u)$ be as above.
  \begin{enumerate}
  \item (\emph{Global concavity.}) If
    $\dfrac{1}{n+2}\le\lambda<\dfrac{1}{n+1}$, then $L(u)<0$
    for all $u>0$, so $f$ is globally logarithmically strictly
    concave on $(0,\infty)$.
  \item (\emph{Logarithmic affinity.}) If $\lambda=\dfrac{1}{n+1}$,
    then $f\equiv1$, so $L\equiv0$.
  \item (\emph{Local failure of convexity at the origin.}) If
    $\dfrac{n+1}{n^2+4n+5}<\lambda<\dfrac{1}{n+2}$, then $L(0)<0$.
  \end{enumerate}
\end{theorem}

\begin{proof}
  Taking the derivative of the integrand in $f'(u)$ with respect to
  $u$ gives
  \[
    \frac{\partial}{\partial u}
    \bigl[v(1-v)^n(1+\lambda uv)^{\alpha-1}\bigr]
    =(\alpha-1)\lambda v^2(1-v)^n(1+\lambda uv)^{\alpha-2}.
  \]
  By Lemma~\ref{lem:domination} with $\beta=\alpha-2$, the same
  bounds on any compact $[a,b]\subset(-1/\lambda,\infty)$ give
  $|(1+\lambda uv)^{\alpha-2}|\le K_{\alpha-2}<\infty$. Hence
  Leibniz's rule applies for $f''$ as well. Using
  $\alpha\lambda=1-(n+1)\lambda$ and
  $(\alpha-1)\lambda=1-(n+2)\lambda$, we have
  \begin{align}
    f'(u)  &= (n+1)\bigl(1-(n+1)\lambda\bigr)
              \int_0^1 v(1-v)^n(1+\lambda uv)^{\alpha-1}\dee v,
              \label{eq:f1}\\[4pt]
    f''(u) &= (n+1)\bigl(1-(n+1)\lambda\bigr)\bigl(1-(n+2)\lambda\bigr)
              \int_0^1 v^2(1-v)^n(1+\lambda uv)^{\alpha-2}\dee v.
              \label{eq:f2}
  \end{align}
  We now treat each case.
  \begin{enumerate}
  \item Suppose $\dfrac{1}{n+2}\le\lambda<\dfrac{1}{n+1}$. We consider
  two subcases.

  \emph{Subcase $\lambda>\frac{1}{n+2}$}: $1-(n+1)\lambda>0$ and
  $1-(n+2)\lambda<0$. Since the integral in~\eqref{eq:f2} is strictly
  positive for all $u\in(-1/\lambda,\infty)$, we have $f''(u)<0$
  throughout the domain. Since $f>0$, $L(u)=(f''f-[f']^2)/f^2<0$
  for all $u>0$.

  \emph{Subcase $\lambda=\frac{1}{n+2}$}: Here
  $\alpha=1/\lambda-n-1=(n+2)-n-1=1$ and
  $1-(n+2)\lambda=0$, so \eqref{eq:f2} gives $f''(u)=0$.
  Therefore $L(u)=-[f'(u)/f(u)]^2$. From~\eqref{eq:f1},
  $1-(n+1)\lambda=1/(n+2)>0$, and since the integrand
  $v(1-v)^n(1+\lambda uv)^{0}$ is strictly positive for
  $v\in(0,1)$, we obtain $f'(u)>0$ for all $u>0$, giving
  $L(u)<0$.

  In both subcases $L(u)<0$ for all $u>0$.

   \item If $\lambda=\frac{1}{n+1}$,
  then by Theorem~\ref{thm:monotonicity-threshold}
  $f\equiv1$, so $L\equiv0$.

  \item  Assume $\dfrac{n+1}{n^2+4n+5}<\lambda<\dfrac{1}{n+2}$.
  The interval is non-empty as the inequality
  $\frac{n+1}{n^2+4n+5}<\frac{1}{n+2}$ is equivalent to
  $(n+1)(n+2)<n^2+4n+5$, i.e.\ $n^2+3n+2<n^2+4n+5$, i.e.\
  $0<n+3$, which holds for all $n\in\No$.

  Since $f(0)=1$, the definition of $L$ gives
  $L(0)=[f''(0)\cdot f(0)-(f'(0))^2]/[f(0)]^2=f''(0)-[f'(0)]^2$.
  Setting $u=0$ in~\eqref{eq:f1} and~\eqref{eq:f2} and using
  $\int_0^1 v^k(1-v)^n\dee v=\frac{k!\,n!}{(n+k+1)!}$:
  \begin{align*}
    f'(0)&=\frac{1-(n+1)\lambda}{n+2},\\
    f''(0)&=\frac{2\bigl(1-(n+1)\lambda\bigr)\bigl(1-(n+2)\lambda\bigr)}{(n+2)(n+3)}.
  \end{align*}
  Setting $P:=1-(n+1)\lambda>0$ and $Q:=1-(n+2)\lambda>0$ yields
  \begin{align*}
    L(0)
    &=\frac{2PQ}{(n+2)(n+3)}-\frac{P^2}{(n+2)^2}
    =\frac{P}{n+2}\cdot\frac{2(n+2)Q-(n+3)P}{(n+2)(n+3)}.
  \end{align*}
  Expanding $N:=2(n+2)Q-(n+3)P$ with $P=1-(n+1)\lambda$,
  $Q=1-(n+2)\lambda$ results in $N=(n+1)-(n^2+4n+5)\lambda$. Hence,
  \[
    L(0)=\frac{P\cdot[(n+1)-(n^2+4n+5)\lambda]}{(n+2)^2(n+3)}.
  \]
  Since $P>0$, we have $L(0)<0$ if and only if
  $\lambda>\frac{n+1}{n^2+4n+5}$, which holds throughout the
  interval $\left(\frac{n+1}{n^2+4n+5},\frac{1}{n+2}\right)$.
  Hence $f$ fails to be globally logarithmically convex on any
  interval containing~$0$.
  \end{enumerate}
\end{proof}

Theorem~\ref{thm:log-convex-failure} establishes a local obstruction
(at $u=0$) to global log-convexity for $\lambda$ in the sub-interval
$\bigl(\frac{n+1}{n^2+4n+5},\frac{1}{n+2}\bigr)$ of the increasing
regime, and a global obstruction for the larger sub-interval
$\bigl[\frac{1}{n+2},\frac{1}{n+1}\bigr)$. The remaining gap,
$\lambda\in\bigl(0,\frac{n+1}{n^2+4n+5}\bigr]$, where $L(0)\ge0$,
is closed by Theorem~\ref{thm:asymptotic-log-concave} via an asymptotic
argument: $L(u)$ must eventually become negative as $u\to\infty$.
Together, Theorems~\ref{thm:log-convex-failure}
and~\ref{thm:asymptotic-log-concave} cover the entire increasing regime
$(0,1/(n+1))$ without overlap (Corollary~\ref{cor:unified-log-concave}).
The following lemma provides the key asymptotic tool.

\begin{lemma}[Asymptotic factorisation]\label{lem:asymptotic-factor}
  Let $\lambda\in(0,1/(n+1))$, so that $\alpha=1/\lambda-n-1>0$.
  Define $I_j(u):=\int_0^1 v^{j}(1-v)^n(1+\lambda uv)^{\alpha-j}\dee v$
  for $j\in\{0,1,2\}$ and $G_n:=B(n+1,\alpha+1)>0$.
  Then for each $j$,
  \[
    I_j(u)=(1+\lambda u)^{\alpha-j}\bigl[G_n+o(1)\bigr]
    \qquad\text{as }u\to\infty.
  \]
\end{lemma}

\begin{proof}
Substitute $w=1-v$ to write
$I_j(u)=\int_0^1(1-w)^{j}w^n(1+\lambda u(1-w))^{\alpha-j}\dee w$.
Factor out $(1+\lambda u)^{\alpha-j}$ by writing
$1+\lambda u(1-w)=(1+\lambda u)(1-rw)$ where
$r:=\lambda u/(1+\lambda u)\in[0,1)$. This gives
$I_j(u)=(1+\lambda u)^{\alpha-j}h_j(r)$ with
$h_j(r):=\int_0^1(1-w)^{j}w^n(1-rw)^{\alpha-j}\dee w$.

\textit{Uniform domination.}
We show $(1-w)^j(1-rw)^{\alpha-j}\le1$ for all
$w\in[0,1]$ and $r\in[0,1)$, uniformly in $j\in\{0,1,2\}$.
Since $\alpha>0$:
\begin{itemize}
\item If $\alpha-j\ge0$: $(1-w)^j\le1$ and $(1-rw)^{\alpha-j}\le1$.
  (as both bases lie in $[0,1]$ and exponents are non-negative).
\item If $\alpha-j<0$: since $0\le r<1$ we have $rw\le w$, hence
  $1-rw\ge1-w\ge0$. The map $s\mapsto s^{\alpha-j}$ is decreasing
  on $(0,\infty)$, so $(1-rw)^{\alpha-j}\le(1-w)^{\alpha-j}$, giving
  $(1-w)^j(1-rw)^{\alpha-j}\le(1-w)^j(1-w)^{\alpha-j}=(1-w)^\alpha\le1$,
  where the last inequality uses $\alpha>0$ and $1-w\in[0,1]$.
\end{itemize}
In both cases $(1-w)^j(1-rw)^{\alpha-j}\le1$, so the integrand of
$h_j(r)$ is bounded above by $w^n$, which is integrable on $[0,1]$
and independent of both $r$ and $j$.

\textit{Pointwise convergence and conclusion.}
As $u\to\infty$, $r\to1^-$. For each fixed $w\in[0,1)$,
$(1-rw)^{\alpha-j}\to(1-w)^{\alpha-j}$, so
$(1-w)^jw^n(1-rw)^{\alpha-j}\to w^n(1-w)^\alpha$.
At $w=1$: for $j\ge1$ the factor $(1-w)^j=0$ so the integrand
is $0$ for all $r$; for $j=0$ the integrand equals $(1-r)^\alpha\to0$
since $\alpha>0$. Thus convergence holds on all of $[0,1]$.
By dominated convergence,
$h_j(r)\to\int_0^1 w^n(1-w)^\alpha\dee w=G_n$
for every $j\in\{0,1,2\}$, proving the lemma.
\end{proof}

\begin{theorem}[Asymptotic failure of global logarithmic convexity]
  \label{thm:asymptotic-log-concave}
  Let $f(u)=\Tdeg{n}(u)$ and fix any $\lambda\in(0,1/(n+1))$, so
  that $\alpha=1/\lambda-n-1>0$. Let $L(u)=\frac{\dee^2}{\dee u^2}\ln f(u)$
  denote the logarithmic second derivative of $f$ as in
  Section~\ref{sec:log-convexity}. Then
  \[
    \lim_{u\to\infty}u^{2}L(u)=-\alpha<0.
  \]
  In particular, $L(u)<0$ for all sufficiently large $u$, so $f$
  fails to be globally logarithmically convex on $(0,\infty)$ for
  every $\lambda\in(0,1/(n+1))$.
\end{theorem}

\begin{proof}
  \textbf{Step~1.} For $j\in\{0,1,2\}$, define
  \[
    I_j(u):=\int_0^1 v^{j}(1-v)^n(1+\lambda uv)^{\alpha-j}\dee v,
  \]
  so that
  \begin{align*}
  f(u)&=(n+1)I_0(u),\quad
  f'(u)=(n+1)(1-(n+1)\lambda)I_1(u),\\
  f''(u)&=(n+1)(1-(n+1)\lambda)(1-(n+2)\lambda)I_2(u)
  \end{align*}
  by Theorem~\ref{thm:integral-rep}, \eqref{eq:f1}, and~\eqref{eq:f2}.
  By Lemma~\ref{lem:asymptotic-factor}, for each $j\in\{0,1,2\}$,
  $I_j(u)=(1+\lambda u)^{\alpha-j}[G_n+o(1)]$ as $u\to\infty$,
  where $G_n=B(n+1,\alpha+1)>0$ is independent of $j$.

  \textbf{Step~2.}
  Since $(1+\lambda u)/(\lambda u)\to1$ as $u\to\infty$, we have
  $(1+\lambda u)^{\alpha-j}\sim(\lambda u)^{\alpha-j}$ for each
  fixed real exponent $\alpha-j$. Combined with the claim, this
  gives $I_j(u)\sim G_n(\lambda u)^{\alpha-j}$. Using
  $1-(n+1)\lambda=\alpha\lambda$ and
  $1-(n+2)\lambda=(\alpha-1)\lambda$:
  \begin{align*}
    f(u)   &\sim (n+1)G_n(\lambda u)^{\alpha} =: A\,u^{\alpha},\\
    f'(u)  &\sim (n+1)\alpha\lambda\,G_n(\lambda u)^{\alpha-1}
              = \alpha A\,u^{\alpha-1},\\
    f''(u) &\sim (n+1)\alpha(\alpha-1)\lambda^2\,G_n(\lambda u)^{\alpha-2}
              = \alpha(\alpha-1)A\,u^{\alpha-2},
  \end{align*}
  where $A:=(n+1)G_n\lambda^{\alpha}>0$.
  By the multiplicative property of asymptotic equivalence
  (if $g_i\sim c_i u^{p_i}$ then $g_1g_2\sim c_1c_2u^{p_1+p_2}$):
  $f''(u)f(u)\sim\alpha(\alpha-1)A^2u^{2\alpha-2}$,
  $[f'(u)]^2\sim\alpha^2 A^2u^{2\alpha-2}$, and
  $[f(u)]^2\sim A^2u^{2\alpha}$. Therefore,
  \[
    L(u)=\frac{f''f-[f']^2}{f^2}
    \sim\frac{[\alpha(\alpha-1)-\alpha^2]A^2\,u^{2\alpha-2}}
             {A^2\,u^{2\alpha}}
    =\frac{-\alpha}{u^{2}}\qquad(u\to\infty),
  \]
  so $\lim_{u\to\infty}u^2L(u)=-\alpha<0$, and $L(u)<0$ for all
  sufficiently large $u$.
\end{proof}

Consequently, the threshold $\frac{n+1}{n^2+4n+5}$ in
Theorem~\ref{thm:log-convex-failure} cannot be improved to a
statement of global logarithmic convexity at any positive $\lambda$
in the increasing regime.

\begin{remark}[Rate of convergence]
  The limit $u^2L(u)\to-\alpha$ in
  Theorem~\ref{thm:asymptotic-log-concave} is attained at a rate
  $u^2L(u)+\alpha=O(u^{-1})$ as $u\to\infty$, because $h_j(r)-G_n=O(1-r)=O(u^{-1})$ as $u\to\infty$
  (since $r=\lambda u/(1+\lambda u)=1-O(u^{-1})$).
  In particular, when $\alpha$ is small (i.e.\ $\lambda$ is near
  $1/(n+1)$) the convergence is slow, and numerically one needs
  $u\gg 1/\alpha$ to observe $u^2L(u)$ close to $-\alpha$;
  this does not affect the theoretical validity of the theorem.
\end{remark}

\begin{remark}
  Theorem~\ref{thm:asymptotic-log-concave} shows that even for
  $\lambda\le\frac{n+1}{n^2+4n+5}$, where $L(0)\ge0$ and $f$ is
  locally logarithmically convex at the origin (in the sense
  that $L(0)\ge0$), global logarithmic convexity on all of
  $(0,\infty)$ is impossible, because $L(u)$ must eventually turn
  strictly negative as $u\to\infty$.
\end{remark}

\begin{corollary}[Failure of global logarithmic convexity throughout
  the increasing regime]\label{cor:unified-log-concave}
  For every $\lambda\in(0,1/(n+1))$, $f=\Tdeg{n}$ fails to be
  globally logarithmically convex on $(0,\infty)$.
\end{corollary}

\begin{proof}
  The interval $(0,1/(n+1))$ splits as
  $(0,1/(n+2))\cup[1/(n+2),1/(n+1))$. For
  $\lambda\in[1/(n+2),1/(n+1))$, Theorem~\ref{thm:log-convex-failure}
  (Case~1) gives $L(u)<0$ for all $u>0$. For
  $\lambda\in(0,1/(n+2))$,
  Theorem~\ref{thm:asymptotic-log-concave} gives $L(u)<0$ for all
  sufficiently large $u$. In both cases $f$ is not globally
  logarithmically convex.
\end{proof}

\begin{remark}
  Corollary~\ref{cor:unified-log-concave} is sharp in scope. It is
  restricted to the increasing regime $\lambda\in(0,1/(n+1))$.
  At the single boundary point $\lambda=1/(n+1)$, $f\equiv1$
  identically, so $L\equiv0$ and $f$ is (trivially, weakly) globally
  logarithmically convex there. The decreasing regime
  $\lambda>1/(n+1)$ is left open.
\end{remark}

%% ==================================================================
\section{Absolute Monotonicity}\label{sec:abs-monotonicity}
%% ==================================================================

The notion of absolutely monotonic functions goes back to Widder~\cite{Widder1941}, whose work establishes its central role in the theory of Bernstein representations and Laplace transforms.

\begin{definition}\label{def:abs-monotone}
  A function $g\colon[0,\infty)\to\R$ is called \emph{absolutely
  monotonic} if $g^{(k)}(u)\ge0$ for every $u\ge0$ and every
  $k\in\No$.
\end{definition}

Before stating the main theorem, we introduce the standard
falling factorial. For $\alpha\in\R$ and $k\in\No$, set
\[
  (\alpha)_k=\alpha(\alpha-1)\cdots(\alpha-k+1),
  \qquad(\alpha)_0=1.
\]
Note the distinction from the degenerate falling factorial
$\dfact{x}{n}$ of Definition~\ref{def:factorials}: the latter
depends on the parameter $\lambda$, whereas $(\alpha)_k$ does not.

\begin{theorem}[Necessary and sufficient condition for absolute
  monotonicity]
  \label{thm:abs-monotonicity}
  Fix $\lambda>0$, $n\in\No$, and $\alpha=\tfrac{1}{\lambda}-n-1$.
  Then $\Tdeg{n}$ is absolutely monotonic on $[0,\infty)$ if and
  only if $\alpha\in\No$.
\end{theorem}

\begin{proof}
  We prove by induction on $k$ that for every $k\in\No$,
  \begin{equation}\label{eq:kth-T}
    \Tdeg{n}^{(k)}(u)
    =(n+1)(\alpha)_k\lambda^k
    \int_0^1 v^k(1-v)^n(1+\lambda uv)^{\alpha-k}\dee v.
  \end{equation}

  \emph{Base case ($k=0$).} Formula~\eqref{eq:kth-T} reduces to
  the integral representation of Theorem~\ref{thm:integral-rep}.

  \emph{Inductive step.} Assume \eqref{eq:kth-T} holds for some
  $k\ge0$. By Lemma~\ref{lem:domination} applied with exponent
  $\alpha-k-1$, Leibniz's rule gives
  \[
    \Tdeg{n}^{(k+1)}(u)
    =(n+1)(\alpha)_{k+1}\lambda^{k+1}
    \int_0^1 v^{k+1}(1-v)^n(1+\lambda uv)^{\alpha-k-1}\dee v,
  \]
  which is~\eqref{eq:kth-T} at $k+1$, completing the induction.

  For $u\ge0$, the integrand $v^k(1-v)^n(1+\lambda uv)^{\alpha-k}$
  is strictly positive for $v\in(0,1)$, so the integral is strictly
  positive. Since $\lambda>0$ and $(n+1)(\alpha)_k\lambda^k$
  is a real constant (depending only on the parameters, not on $u$),
  the sign of $\Tdeg{n}^{(k)}(u)$ for $u\ge0$ is exactly the sign
  of $(\alpha)_k$: positive if $(\alpha)_k>0$, zero if
  $(\alpha)_k=0$, and negative if $(\alpha)_k<0$. No cancellation
  with the integral term is possible since the integral is always
  strictly positive.

  \smallskip\noindent\textbf{Sufficiency ($\alpha\in\No$).}
  Let $\alpha=m\in\No$.
  \begin{itemize}
  \item $k=0$: $(\alpha)_0=1>0$.
  \item $1\le k\le m$: each factor $m-j\ge1>0$ for $j=0,\dotsc,k-1$,
    so $(\alpha)_k\ge1>0$.
  \item $k>m$: $(\alpha)_k$ contains the factor $m-m=0$
    (at $j=m$), giving $(\alpha)_k=0$.
  \end{itemize}
  Hence $\Tdeg{n}^{(k)}(u)\ge0$ for all $k\in\No$, $u\ge0$.

  \smallskip\noindent\textbf{Necessity ($\alpha\notin\No$).}
  Since $\alpha\notin\No$, no factor $\alpha-k$ ($k\ge0$) vanishes.

  \emph{Case~(a): $\alpha<0$.} $(\alpha)_1=\alpha<0$, so
  $\Tdeg{n}'(u)<0$: not absolutely monotonic.

  \emph{Case~(b): $\alpha>0$, $\alpha\notin\No$.}
  Let $m=\lfloor\alpha\rfloor$, so $m<\alpha<m+1$.
  Then $\alpha-m>0$ and $\alpha-m-1<0$; every factor in
  $(\alpha)_{m+1}=\prod_{j=0}^{m}(\alpha-j)$ is positive, so
  $(\alpha)_{m+1}>0$; but $(\alpha)_{m+2}=(\alpha-m-1)(\alpha)_{m+1}<0$.
  Hence $\Tdeg{n}^{(m+1)}(u)>0$ but $\Tdeg{n}^{(m+2)}(u)<0$.

  In both cases the failure of absolute monotonicity is confirmed.
\end{proof}

\begin{corollary}[Failure of absolute monotonicity for almost every
  $\lambda$]\label{cor:ae-failure}
  For fixed $n\in\No$, the set of $\lambda>0$ for which $\Tdeg{n}$
  is absolutely monotonic on $[0,\infty)$ is exactly
  $\bigl\{\tfrac{1}{n+1+m}:m\in\No\bigr\}$, a countably infinite
  set with Lebesgue measure zero. Consequently, $\Tdeg{n}$ fails to
  be absolutely monotonic for Lebesgue-almost every $\lambda>0$.
\end{corollary}

\begin{proof}
  By Theorem~\ref{thm:abs-monotonicity}, absolute monotonicity holds
  iff $\alpha=1/\lambda-n-1\in\No$, i.e.\ $\lambda=1/(n+1+m)$ for
  some $m\in\No$. This countable set has Lebesgue measure zero.
\end{proof}

\begin{remark}
  The set $\{1/(n+1+m):m\in\No\}$ is a strictly decreasing
  sequence converging to $0$, so it is a discrete subset of
  $(0,\infty)$ with single accumulation point $\lambda=0$.
  This stands in sharp contrast to the classical case: for
  $T_n[\e^u]$, absolute monotonicity holds for every
  $n\in\No$ and every $\lambda>0$ (Qi~\cite{Qi2024}), with no
  restriction whatsoever on parameters. The degenerate analogue
  thus exhibits a qualitative phase transition: a dense open
  set of parameters (the complement of our measure-zero set)
  fails absolute monotonicity entirely. The discreteness is
  intrinsic: it arises from the requirement that
  $1/\lambda-n-1\in\No$, a Diophantine-type condition that
  generically fails, and that is consistent with the recovery
  of all classical properties as $\lambda\to0^+$
  (Theorem~\ref{thm:classical-limit}).
\end{remark}

%% ==================================================================
\section{Sharp Error Bounds}\label{sec:bounds}
%% ==================================================================

\begin{theorem}[Two-sided bounds; pointwise sharp at origin,
  not asymptotically sharp]%
  \label{thm:bounds}
  Let $\alpha_n=\tfrac{1}{\lambda}-n-1$.
  \begin{enumerate}
  \item If $\lambda=\frac{1}{n+1}$ (i.e.\ $\alpha_n=0$), then $\Tdeg{n}(u)\equiv1$.

  \item \textbf{For $u>0$:}
    \begin{enumerate}
    \item If $\lambda<\frac{1}{n+1}$ (i.e.\ $\alpha_n>0$), then
      $\;1\le\Tdeg{n}(u)\le(1+\lambda u)^{\alpha_n}$.
    \item If $\lambda>\frac{1}{n+1}$ (i.e.\ $\alpha_n<0$), then
      $\;(1+\lambda u)^{\alpha_n}\le\Tdeg{n}(u)\le1$.
    \end{enumerate}

  \item \textbf{For $u\in(-1/\lambda,0)$:}
    \begin{enumerate}
    \item If $\lambda<\frac{1}{n+1}$ (i.e.\ $\alpha_n>0$), then
      $\;(1+\lambda u)^{\alpha_n}\le\Tdeg{n}(u)\le1$.
    \item If $\lambda>\frac{1}{n+1}$ (i.e.\ $\alpha_n<0$), then
      $\;1\le\Tdeg{n}(u)\le(1+\lambda u)^{\alpha_n}$.
    \end{enumerate}
  \end{enumerate}
\end{theorem}

\begin{proof}
  From~\eqref{eq:int-rep-formula},
  $\Tdeg{n}(u)=(n+1)\int_0^1(1-v)^n(1+\lambda uv)^{\alpha_n}\dee v$.
  The weight $(n+1)(1-v)^n\dee v$ is a probability measure on
  $[0,1]$ (since $(n+1)\int_0^1(1-v)^n\dee v=1$).
\begin{enumerate}
\item For $\alpha_n=0$, $(1+\lambda uv)^0=1$ so $\Tdeg{n}(u)=1$.
\item\begin{enumerate}
\item For $u>0$, $v\in[0,1]$: $1\le1+\lambda uv\le1+\lambda u$. Since
  $\alpha_n>0$ the map $s\mapsto s^{\alpha_n}$ is increasing, giving
  $1\le(1+\lambda uv)^{\alpha_n}\le(1+\lambda u)^{\alpha_n}$.
  Integrating against the probability measure $(n+1)(1-v)^n\dee v$
  (which preserves weak inequalities between constants and
  functions of $v$) gives $1\le\Tdeg{n}(u)\le(1+\lambda u)^{\alpha_n}$.
\item Same chain, but $\alpha_n<0$ reverses it:
  $(1+\lambda u)^{\alpha_n}\le(1+\lambda uv)^{\alpha_n}\le1$, hence
  $(1+\lambda u)^{\alpha_n}\le\Tdeg{n}(u)\le1$.
\end{enumerate}
\item\begin{enumerate}
\item For $u<0$, $v\in[0,1]$: $\lambda u\le\lambda uv\le0$, so
  $0<1+\lambda u\le1+\lambda uv\le1$ (using $u>-1/\lambda$).
  Since $\alpha_n>0$:
  $(1+\lambda u)^{\alpha_n}\le(1+\lambda uv)^{\alpha_n}\le1$,
  hence $(1+\lambda u)^{\alpha_n}\le\Tdeg{n}(u)\le1$.
\item Same chain with $\alpha_n<0$:
  $1\le(1+\lambda uv)^{\alpha_n}\le(1+\lambda u)^{\alpha_n}$,
  hence $1\le\Tdeg{n}(u)\le(1+\lambda u)^{\alpha_n}$.
\end{enumerate}
\end{enumerate}
\end{proof}

\begin{remark}
  At $u=0$ all three quantities equal $1$ in every case, confirming
  that the bounds are pointwise sharp at the origin. However, they
  are not asymptotically sharp as $u\to\infty$. By the asymptotic
  $\Tdeg{n}(u)\sim(n+1)G_n(\lambda u)^{\alpha_n}$ (where
  $G_n=B(n+1,\alpha_n+1)$) we have
  \[
    \lim_{u\to\infty}\frac{\Tdeg{n}(u)}{(1+\lambda u)^{\alpha_n}}
    =(n+1)\,B(n+1,\alpha_n+1)
    =\frac{(n+1)!\,\Gamma(\alpha_n+1)}{\Gamma(n+\alpha_n+2)}<1
  \]
  for every $n\ge0$ and $\alpha_n>0$. Hence the upper bound in
  Case~2(a) overestimates $\Tdeg{n}(u)$ by a constant factor as
  $u\to\infty$.
\end{remark}

\begin{corollary}[Explicit truncation-error bounds for $\expdeg$]
  \label{cor:truncation-error}
  Let $n\in\No$ and $\lambda\in(0,1/(n+1))$, so that
  $\alpha_n>0$. Then:
  \begin{enumerate}
  \item For $u>0$,
    $\;\dfrac{\dfact{1}{n+1}}{(n+1)!}\,u^{n+1}
      \le \Rdeg{n}(u) \le
      \dfrac{\dfact{1}{n+1}}{(n+1)!}\,u^{n+1}(1+\lambda u)^{\alpha_n}$.
  \item For $u\in(-1/\lambda,0)$ with $n$ odd
    (so $n+1$ is even and $u^{n+1}>0$),
    $\;\dfrac{\dfact{1}{n+1}}{(n+1)!}\,u^{n+1}(1+\lambda u)^{\alpha_n}
      \le \Rdeg{n}(u) \le
      \dfrac{\dfact{1}{n+1}}{(n+1)!}\,u^{n+1}$.
  \item For $u\in(-1/\lambda,0)$ with $n$ even
    (so $n+1$ is odd and $u^{n+1}<0$),
    $\;\dfrac{\dfact{1}{n+1}}{(n+1)!}\,u^{n+1}
      \le \Rdeg{n}(u) \le
      \dfrac{\dfact{1}{n+1}}{(n+1)!}\,u^{n+1}(1+\lambda u)^{\alpha_n}$.
  \end{enumerate}
\end{corollary}

\begin{proof}
 For $\lambda\in(0,1/(n+1))$, every factor in
  $\dfact{1}{n+1}=\prod_{k=0}^{n}(1-k\lambda)$ is positive, hence $\dfact{1}{n+1}>0$. By Definition~\ref{def:degenerate-remainder},
  for $u\ne0$, we have
  \[
    \Rdeg{n}(u)=\frac{\dfact{1}{n+1}}{(n+1)!}\,u^{n+1}\,\Tdeg{n}(u).
  \]
\begin{enumerate}
\item Here $u^{n+1}>0$, so
  $\dfact{1}{n+1}u^{n+1}/(n+1)!>0$. Multiplying the bounds of
  Theorem~\ref{thm:bounds}, Case~(2a) by this positive quantity preserves
  both inequalities, giving the stated bound.
\item Theorem~\ref{thm:bounds},
  Case~(3a) gives $(1+\lambda u)^{\alpha_n}\le\Tdeg{n}(u)\le1$.
  If $n$ is odd then $n+1$ is even, so $u^{n+1}>0$ for $u<0$.
  Since $\dfact{1}{n+1}>0$, multiplying by the positive factor
  $\dfact{1}{n+1}u^{n+1}/(n+1)!$ preserves the inequality direction,
  giving Case~(2).
\item If $n$ is even, $u^{n+1}<0$ and multiplying
  reverses the inequality direction, giving Case~(3).
\end{enumerate}
\end{proof}

\begin{remark}
  Corollary~\ref{cor:truncation-error} provides explicit two-sided
  bounds on the Maclaurin truncation error
  $\Rdeg{n}(u)=\e_\lambda(u)-\sum_{j=0}^n\dfact{1}{j}\,u^j/j!$,
  connecting the normalized-remainder analysis to the
  approximation-theoretic question of how accurately the first
  $n+1$ terms approximate $\e_\lambda(u)$.
\end{remark}

%% ==================================================================
\section{Classical Limit}\label{sec:classical-limit}
%% ==================================================================

\begin{theorem}[Classical limit]\label{thm:classical-limit}
  For each fixed $n\in\No$ and $u\in\R$,
  \[
    \lim_{\lambda\to0^{+}}\Tdeg{n}(u)
    =T_n[\e^{u}]
    :=(n+1)\int_0^1(1-v)^n\e^{uv}\dee v.
  \]
  Moreover, for all $\lambda\in(0,1/(n+1))$, $\Tdeg{n}(u)$ is
  strictly increasing in $u$.
\end{theorem}

\begin{proof}
  Fix $u\in\R$ and choose
  $\lambda_0:=\min\!\bigl(\frac{1}{4(n+2)},\frac{1}{4|u|+1}\bigr)>0$.
  For all $\lambda\in(0,\lambda_0]$ and $v\in[0,1]$:
  $|\lambda uv|\le\lambda_0|u|\le\frac{|u|}{4|u|+1}<\frac{1}{4}$,
  so
  \begin{equation}\label{eq:lb-v9}
    1+\lambda uv\ge1-|\lambda uv|\ge\tfrac{3}{4}>0.
  \end{equation}
  The standard inequality $\ln(1+x)\le x$ for $x>-1$, applied to
  $x=\lambda uv$ (valid since $1+\lambda uv\ge3/4>0$), gives
  \[
    \tfrac{1}{\lambda}\ln(1+\lambda uv)\le uv,
    \quad\text{hence}\quad
    (1+\lambda uv)^{1/\lambda}\le\e^{uv}\le\e^{|u|}.
  \]
  Combined with $(1+\lambda uv)^{-n-1}\le(3/4)^{-(n+1)}=(4/3)^{n+1}
  =:C_n$, we obtain the dominating bound
  \[
    g_\lambda(v):=(1-v)^n(1+\lambda uv)^{1/\lambda-n-1}
    \le C_n\e^{|u|},
  \]
  valid for all $\lambda\in(0,\lambda_0]$ and $v\in[0,1]$.

  By the sequential characterization of limits, fix any sequence
  $\lambda_k\to0^+$ with $\lambda_k\in(0,\lambda_0]$. Then
  $g_{\lambda_k}(v)\to(1-v)^n\e^{uv}$ pointwise (using
  $(1+\lambda_k uv)^{1/\lambda_k}\to\e^{uv}$), and the domination
  bound is uniform in $k$. By dominated convergence,
  \[
    \lim_{\lambda\to0^{+}}\Tdeg{n}(u)
    =(n+1)\int_0^1(1-v)^n\e^{uv}\dee v=T_n[\e^{u}].
  \]
  Strict monotonicity for $\lambda\in(0,1/(n+1))$ follows from
  Theorem~\ref{thm:monotonicity-threshold}(1).
\end{proof}

%% ==================================================================
\section{Conclusion}
\label{sec:conclusion}
%% ==================================================================

We have studied the normalized remainder $\Tdeg{n}$ for the degenerate
exponential function and revealed a rich structure of threshold
phenomena with no classical parallel. While the classical normalized
remainder $T_n[\e^u]$ is globally increasing and globally
logarithmically convex on $(0,\infty)$ for every $n$
(Li--Qi~\cite{LiQi2025}), the degenerate remainder transitions
sharply between increasing and decreasing at $\lambda=1/(n+1)$, and
is never globally logarithmically convex for any $\lambda$ in the
increasing regime, no matter how small
(Corollary~\ref{cor:unified-log-concave}).
The underlying mechanism is a qualitative change in the remainder's
asymptotic growth: where the classical remainder grows like $\e^u$,
the degenerate one grows like a power $u^\alpha$, and this power-type
behavior is incompatible with global logarithmic convexity. Absolute
monotonicity is similarly restricted, holding only for a countable,
measure-zero set of degeneracy parameters
(Corollary~\ref{cor:ae-failure}). The two-sided bounds of
Theorem~\ref{thm:bounds} yield explicit truncation-error estimates
for the degenerate exponential series
(Corollary~\ref{cor:truncation-error}).

To summarise the main theorem chain: the integral representation
(Theorem~\ref{thm:integral-rep}) is the single technical foundation
from which all main results follow. From it, the monotonicity
threshold (Theorem~\ref{thm:monotonicity-threshold}), the local
log-convexity failure (Theorem~\ref{thm:log-convex-failure}), and
the asymptotic global failure (Theorem~\ref{thm:asymptotic-log-concave})
all follow by differentiating under the integral sign and extracting
the leading power-type asymptotics via Lemma~\ref{lem:asymptotic-factor}.
The absolute monotonicity classification (Theorem~\ref{thm:abs-monotonicity})
and the two-sided bounds (Theorem~\ref{thm:bounds}) are further
independent consequences of the same representation.

The case $\lambda<0$, where the domain is left-unbounded and the
asymptotic behavior is governed by a boundary limit rather than
$u\to\infty$, presents different technical challenges and will be
taken up separately. More generally, it would be interesting to
know whether the power-type asymptotic growth that drives our main
results is a broader feature of normalized remainders for degenerate
special functions built from $\expdeg$.

\section*{Declaration of Competing Interest}
The authors declare that they have no known competing financial interests or personal relationships that could have appeared to influence the work reported in this paper.

\section*{Funding}
This research did not receive any specific grant from funding agencies in the public, commercial, or not-for-profit sectors.

%% ==================================================================


\begin{thebibliography}{99}
%% ==================================================================


\bibitem{BaoAgarwalQiDu2024}
Z.-H.~Bao, R.~P.~Agarwal, F.~Qi, and W.-S.~Du,
\emph{Some properties on normalized tails of Maclaurin power series expansion of exponential function},
Symmetry \textbf{16} (2024), no.~8, Article 989, 15~pp.


\bibitem{Carlitz1956}
L.~Carlitz,
\emph{A degenerate Staudt-Clausen theorem},
Arch.\ Math.\ (Basel) \textbf{7} (1956), 28--33.

\bibitem{Carlitz1979}
L.~Carlitz,
\emph{Degenerate Stirling, Bernoulli and Eulerian numbers},
Utilitas Math.\ \textbf{15} (1979), 51--88.

\bibitem{Cesarano2024}
C.~Cesarano, Y.~Quintana, and W.~Ramirez,
\emph{Degenerate versions of hypergeometric Bernoulli-Euler
  polynomials},
Lobachevskii J.\ Math.\ \textbf{45} (2024), 3509--3520.

\bibitem{ChenKimKimLeeLee2024}
L.~Chen, T.~Kim, D.~S.~Kim, H.~Lee, and S.-H.~Lee,
\emph{Probabilistic degenerate central Bell polynomials},
Math.\ Comput.\ Model.\ Dyn.\ Syst.\ \textbf{30} (2024), no.~1,
  523--542.

\bibitem{HernandezPeraltaQuintana2023}
J.~Hern\'andez, D.~Peralta, and Y.~Quintana,
\emph{A look at generalized degenerate Bernoulli and Euler matrices},
Mathematics \textbf{11} (2023), no.~12, Article 2731, 17~pp.

\bibitem{KimKim2017Laplace}
T.~Kim and D.~S.~Kim,
\emph{Degenerate Laplace transform and degenerate gamma function},
Russ.\ J.\ Math.\ Phys.\ \textbf{24} (2017), no.~2, 241--248.

\bibitem{KimKim2020JMAA}
T.~Kim and D.~S.~Kim,
\emph{Degenerate polyexponential functions and degenerate Bell
  polynomials},
J.\ Math.\ Anal.\ Appl.\ \textbf{487} (2020), no.~2, Article
  124017, 15~pp.

\bibitem{KimKimKwonLee2020}
T.~Kim, D.~S.~Kim, J.~Kwon, and H.~Lee,
\emph{Degenerate polyexponential functions and type~2 degenerate
  poly-Bernoulli numbers and polynomials},
Adv.\ Difference Equ.\ \textbf{2020} (2020), Paper No.~168.

\bibitem{KimKim2021Sheffer}
D.~S.~Kim and T.~Kim,
\emph{Degenerate Sheffer sequence and $\lambda$-Sheffer sequence},
J.\ Math.\ Anal.\ Appl.\ \textbf{493} (2021), no.~1, Article
  124521, 17~pp.

\bibitem{KimKim2022}
T.~Kim and D.~S.~Kim,
\emph{On some degenerate differential and degenerate difference
  operators},
Russ.\ J.\ Math.\ Phys.\ \textbf{29} (2022), 37--46.


\bibitem{LiQi2025}
Y.-W.~Li and F.~Qi,
\emph{Elegant proofs for properties of normalized remainders of
  Maclaurin power series expansion of exponential function},
Mathematica Slovaca \textbf{75} (2025), no.~5, 1035--1044.

\bibitem{LiuLongQi2025}
X.-L.~Liu, H.-X.~Long, and F.~Qi,
\emph{Monotonicity results of ratio between two normalized remainders
  of Maclaurin series expansion for square of tangent function},
Mathematica Slovaca \textbf{75} (2025), no.~3, 699--705.

\bibitem{Qi2024}
F.~Qi,
\emph{Absolute monotonicity of normalized tail of power series
  expansion of exponential function},
Mathematics \textbf{12} (2024), no.~18, Article 2859, 11~pp.

\bibitem{QiSurvey2025}
F.~Qi,
\emph{Theory of normalized remainders in Taylor series expansions},
arXiv preprint arXiv:2512.15797 (2025; revised 2026).

\bibitem{QiGuo2021JMAA}
F.~Qi and B.-N.~Guo,
\emph{From inequalities involving exponential functions and sums to
  logarithmically complete monotonicity of ratios of gamma functions},
J.\ Math.\ Anal.\ Appl.\ \textbf{493} (2021), no.~1, Article
  124478, 19~pp.

\bibitem{WanQi2024}
A.~Wan and F.~Qi,
\emph{Power series expansion, decreasing property, and concavity
  related to logarithm of normalized tail of power series expansion
  of cosine},
Electron.\ Res.\ Arch.\ \textbf{32} (2024), no.~5, 3130--3144.

\bibitem{WangQi2024Sine}
F.~Wang and F.~Qi,
\emph{Power series expansion and decreasing property related to
  normalized remainders of power series expansion of sine},
Filomat \textbf{38} (2024), no.~29, 10447--10462.

\bibitem{Widder1941}
D.~V.~Widder,
\emph{The Laplace Transform},
Princeton Mathematical Series, vol.~6,
Princeton University Press, Princeton, NJ, 1941.

\bibitem{ZhangQi2025Tan}
G.-Z.~Zhang and F.~Qi,
\emph{On convexity and power series expansion for logarithm of normalized tail of power series expansion for square of tangent},
J.\ Math.\ Inequal.\ \textbf{18} (2024), no.~3, 937--952.

\bibitem{ZhangYangQi2025}
G.-Z.~Zhang, Z.-H.~Yang, and F.~Qi,
\emph{On normalized tails of series expansion of generating
  function of Bernoulli numbers},
Proc.\ Amer.\ Math.\ Soc.\ \textbf{153} (2025), no.~1, 131--141.

\bibitem{ZhangQi2025}
T.~Zhang and F.~Qi,
\emph{Decreasing ratio between two normalized remainders of
  Maclaurin series expansion of exponential function},
AIMS Math.\ \textbf{10} (2025), no.~6, 14739--14756.

\bibitem{ZhangQi2025Gamma}
J.~Zhang and F.~Qi,
\emph{Some properties of normalized remainders of the Maclaurin
  expansion for a function originating from an integral representation
  of the reciprocal of the gamma function},
Math.\ Inequal.\ Appl.\ \textbf{28} (2025), no.~2, 343--354.

\end{thebibliography}
\end{document}